\documentclass{amsart}

\newtheorem{theorem}{Theorem}[section]
\newtheorem{lemma}[theorem]{Lemma}

\newtheorem{corollary}[theorem]{Corollary}
\newtheorem{proposition}[theorem]{Proposition}

\usepackage{graphicx,amssymb}
\usepackage[T1]{fontenc}

\theoremstyle{definition}

\theoremstyle{remark}

\numberwithin{equation}{section}

\begin{document}

\title[Cyclic branched covers of alternating knots]{Cyclic branched covers
of alternating knots and $L$-spaces}

\author{Masakazu Teragaito}
\address{Department of Mathematics and Mathematics Education, Hiroshima University,
1-1-1 Kagamiyama, Higashi-hiroshima, Japan 739-8524.}
%    Current address
%\curraddr{Department of Mathematics and Statistics,
%Case Western Reserve University, Cleveland, Ohio 43403}
\email{teragai@hiroshima-u.ac.jp}
%    \thanks will become a 1st page footnote.
\thanks{The author was partially supported by Japan Society for the Promotion of Science,
Grant-in-Aid for Scientific Research (C), 25400093.
}%

%    General info

%\date{}

%\dedicatory{}


\begin{abstract}
For any alternating knot, it is known that the double branched cover of the $3$-sphere
branched over the knot is an $L$-space.
We show that
the three-fold cyclic branched cover is also an $L$-space
for any genus one alternating knot.
\end{abstract}

\maketitle

%%%%%%%%%%%%%%%%%%%%%%%%%%%%%%%%%%%%
\section{Introduction}\label{sec:intro}

An \textit{$L$-space} $M$ is a rational homology $3$-sphere whose Heegaard Floer homology
$\widehat{HF}(M)$ is a free abelian group of rank equal to $|H_1(M;\mathbb{Z})|$ (\cite{OS1}).
The most typical examples of $L$-spaces are lens spaces.
In recent years, it is recognized that $L$-spaces form an important class of $3$-manifolds.
For example, see \cite{BGW,OS1}.

We consider the problem when cyclic branched covers of the $3$-sphere
branched over a knot or link is an $L$-space.
Toward this direction, Ozsv\'{a}th and Szab\'{o} \cite{OS2} first showed that
the double branched cover of any non-split alternating link (more generally, quasi-alternating link) is an $L$-space.
Peters \cite{P} verified that for a genus one, $2$-bridge knot $C[2m,2n]\ (m,n>0)$ in Conway's notation,
the $d$-fold cyclic branched cover is an $L$-space for any $d\ge 2$, and that
for $C[2m,-2n]\ (m,n>0)$, so is the $3$-fold cyclic branched cover.
For the latter, the same conclusion still holds for the cases $d=4$ (\cite{T})
and $d=5$ (\cite{H}), but it would be false for sufficiently large $d$ (\cite{Hu,T}). 

In this paper, we restrict ourselves to alternating knots. 
As mentioned above, the double branched cover of any alternating knot is an $L$-space.
Then, is the $3$-fold cyclic branched cover an $L$-space?
The answer is positive for genus one, $2$-bridge knots.
However, it is negative, in general.
Let $\Sigma_d(K)$ denote the $d$-fold cyclic branched cover of the $3$-sphere branched over a knot $K$.
By Baldwin \cite{B}, if $K$ is the trefoil, then $\Sigma_d(K)$ is an $L$-space
if and only if $d\le 5$.
This implies that if $K$ is a $(2,m)$-torus knot with $m\ge 7$,
then $\Sigma_3(K)$ is not an $L$-space.
For, $\Sigma_3(K)$ is homeomorphic to the $m$-fold cyclic branched cover of the trefoil.
These torus knots are alternating, but have genus greater than one.
Thus, we will examine the case where alternating knots have genus one.

\begin{theorem}\label{thm:main}
Let $K$ be a $3$-strand pretzel knot $P(2a+1,2b+1,2c+1)$, where $a,b,c>0$.
Then $\Sigma_3(K)$ is an $L$-space.
\end{theorem}

This immediately implies the following.

\begin{corollary}
Let $K$ be a genus one, alternating knot.
Then $\Sigma_3(K)$ is an $L$-space.
\end{corollary}

\begin{proof}
Suppose that $K$ is a genus one, alternating knot.
By \cite[Lemma 3.1]{BZ} (see also \cite{Pa}), $K$ is either a $2$-bridge knot or
a $3$-strand pretzel knot $P(\ell,m,n)$ where $\ell, m, n$ have the same sign.
For a genus one, $2$-bridge knot, Peters \cite{P} shows that
$\Sigma_3(K)$ is an $L$-space.
If $K=P(\ell,m,n)$, then
$\ell,m,n$ are odd by \cite{Ga}.
Thus Theorem \ref{thm:main} gives the conclusion.
\end{proof}

Hence, the rest of paper is devoted to prove Theorem \ref{thm:main}. 
In Section \ref{sec:block},
we describe a link $\mathcal{L}$ whose double branched cover is homeomorphic to
$\Sigma_3(K)$  for $K=P(2a+1,2b+1,2c+1)$.
Then Theorem \ref{thm:main} immediately follows from Theorem \ref{thm:qa}, which claims that the link
$\mathcal{L}$ is quasi-alternating.
Section \ref{sec:det} describes how to calculate determinants of links through Goeritz matrices.
In Section \ref{sec:a1}, we first argue the case where $a=1$.
Section \ref{sec:induction} completes the proof of Theorem \ref{thm:qa} by using
an inductive argument.
The last section contains some remarks.

%%%%%%%%%%%%%%%%%%%%%%%%%%%%%%%%%%%%%%%%%%%%%%%%%%%%%%%%%%%
\section{Quasi-alternating links}\label{sec:block}

Let $K$ be a pretzel knot $P(2a+1,2b+1,2c+1)$ with $a,b,c>0$,
as illustrated in Figure \ref{fig:pretzel}.
Here, each rectangular box consists of vertically right-handed half-twists of indicated number.
This knot has cyclic period two such that its axis is drawn as the horizontal line.
By taking the quotient of this action,
the images of $K$ and the axis give a link $k\cup A$ in Figure \ref{fig:link}.
The central two boxes consist of vertical twists, and the right box consists of
horizontal twists.
Note that each component of this link is unknotted.
Moreover, it is easy to see that two components are interchangeable.

\begin{figure}[tb]
\includegraphics*[scale=0.6]{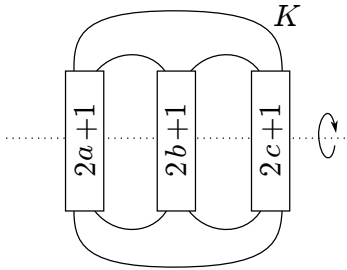}
\caption{A pretzel knot $K=P(2a+1,2b+1,2c+1)$}\label{fig:pretzel}
\end{figure}

\begin{figure}[tb]
\includegraphics*[scale=0.4]{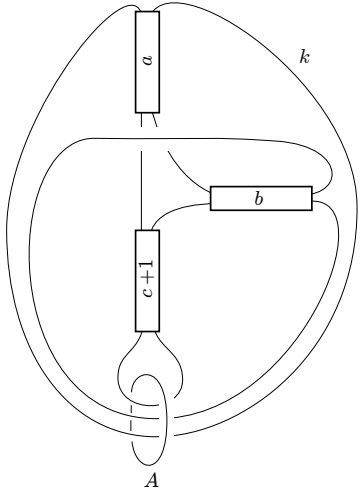}
\caption{The link $k\cup A$}\label{fig:link}
\end{figure}

\begin{proposition}\label{prop:main}
Let $\mathcal{L}$ be the link obtained as the lift of $A$ in $\Sigma_3(k)$, which is the $3$-sphere.
Then $\Sigma_2(\mathcal{L})$ is homeomorphic to $\Sigma_3(K)$.
\end{proposition}

\begin{proof}
Let $M$ be the $\mathbb{Z}_3\oplus \mathbb{Z}_2$ branched cover of $k\cup A$, corresponding
to the map $H_1(S^3-k\cup A)\to \mathbb{Z}_3\oplus \mathbb{Z}_2$ sending
positively oriented meridians of $k$ and $A$ to $(1,0)$ and $(0,1)$, respectively.
Then $M$ is homeomorphic to $\Sigma_2(\mathcal{L})$ and $\Sigma_3(K)$.
\end{proof}

After exchanging the position of $k$ and $A$ in Figure \ref{fig:link},
we still have the same diagram.
Consider $\Sigma_3(k)$.
Then the link $\mathcal{L}$ is as illustrated in Figure \ref{fig:block}.

\begin{figure}[tb]
\includegraphics*[scale=0.4]{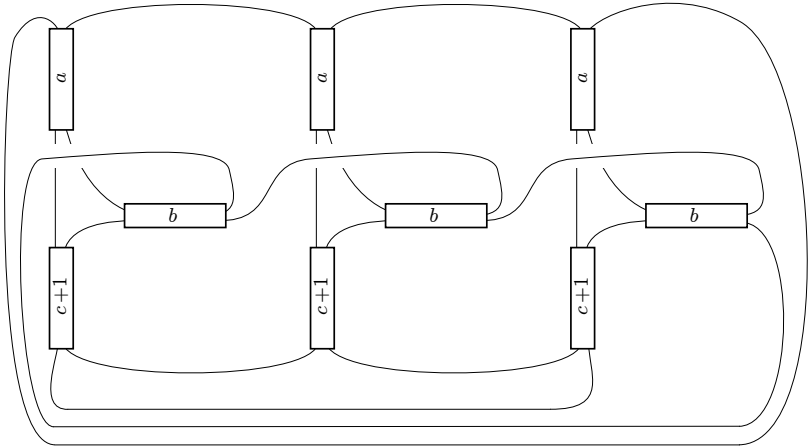}
\caption{The link $\mathcal{L}$}\label{fig:block}
\end{figure}

We recall that the notion of quasi-alternating links \cite{OS2}.
The set of \textit{quasi-alternating links\/} QA is the smallest set of links
satisfying the following.
\begin{itemize}
\item The trivial knot belongs to QA.
\item If a link $L$ has a digram with crossing $c$ such that both of 
two links $L_\infty$ and $L_0$ obtained by smoothing $c$ as in Figure \ref{fig:resolution}
belong to QA, and $\det L=\det L_\infty +\det L_0$, then $L$ belongs to QA.
\end{itemize}

\begin{figure}[tb]
\includegraphics*[scale=0.6]{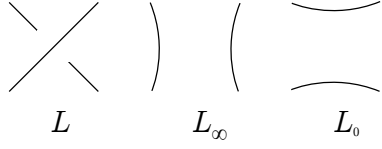}
\caption{Resolutions}\label{fig:resolution}
\end{figure}

As noted in Section \ref{sec:intro}, 
the double branched cover of a quasi-alternating link is an $L$-space, and
any non-split alternating link is quasi-alternating (see \cite{OS2}).

\begin{theorem}\label{thm:qa}
The link $\mathcal{L}$ is quasi-alternating.
Hence, $\Sigma_2(\mathcal{L})$ is an $L$-space.
\end{theorem}

The proof of this theorem is split into Sections \ref{sec:a1} and \ref{sec:induction}.

\begin{proof}[Proof of Theorem \ref{thm:main}]
By Proposition \ref{prop:main}, $\Sigma_3(K)$ is homeomorphic to $\Sigma_2(\mathcal{L})$,
which is an $L$-space by Theorem \ref{thm:qa}.
\end{proof}

%%%%%%%%%%%%%%%%%%%%%%%%%%%%%%%%%%%%%%%%%%%%%%%%%%%%%%%%%%%%%%%%%%%%%%%%
\section{Determinant}\label{sec:det}

To show that the link $\mathcal{L}$ is quasi-alternating,
it is necessary to calculate the determinant of $\mathcal{L}$ and
those of various links arisen from $\mathcal{L}$ by resolutions.
These calculations are done through Goeritz matrices (see \cite{BuZ}).

First, consider the checkerboard coloring of the diagram of $\mathcal{L}$ as in Figure \ref{fig:block}.
The unbounded region is white, and this region will be ignored.
The vertical $a$ right-handed half-twists at the upper left yield
the white regions $\alpha_1,\alpha_4,\dots,\alpha_{3a-2}$ numbered from the top.
Similarly, the white regions $\alpha_2,\alpha_5,\dots,\alpha_{3a-1}$ 
and $\alpha_3,\alpha_6,\dots,\alpha_{3a}$ appear at the upper center and the upper right.
The three white regions just above horizontal $b$ twists are numbered $\alpha_{3a+1},\alpha_{3a+2},\alpha_{3a+3}$ from the left.
Finally, the  white regions $\alpha_{3a+4},\alpha_{3a+5},\alpha_{3a+6}$
are located on the left side of lower twists from the left.
Figure \ref{fig:white} exhibits this numbering convention when $a=1$.

\begin{figure}[tb]
\includegraphics*[scale=0.4]{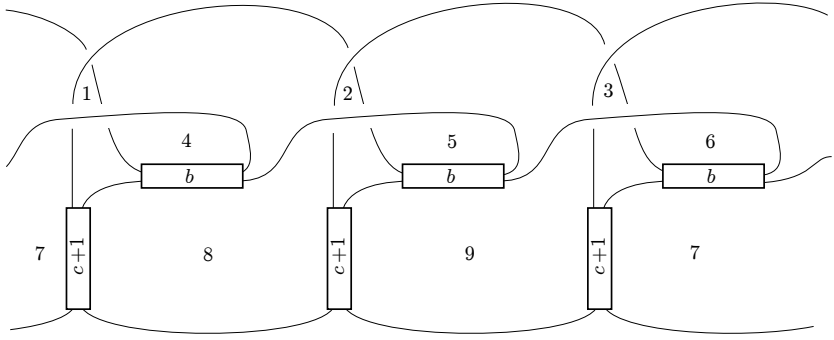}
\caption{The link $\mathcal{L}$ and the white regions when $a=1$}\label{fig:white}
\end{figure}

Figure \ref{fig:sign} shows the convention of sign for each crossing.
The $(3a+6)\times (3a+6)$ Goeritz matrix $G$ is defined as follows.
For $i\ne j$, the $(i,j)$-entry of $G$ is the sum of signs at all the crossings
between the regions $\alpha_i$ and $\alpha_j$.
The $(i,i)$-entry is  $-\sum \mathrm{sign}(c)$, where the sum is over all crossings $c$
around the region $\alpha_i$.
Then it is well known that $|\det G|$ equals to the determinant of $\mathcal{L}$.

\begin{figure}[tb]
\includegraphics*[scale=0.6]{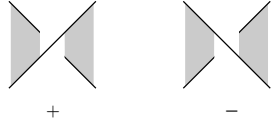}
\caption{Signs of crossing}\label{fig:sign}
\end{figure}

For example, if $a=1$, then the Goeritz matrix $G_1$ is

\[
\left(
\begin{array}{ccc|ccc|ccc}
& \mbox{\large $-I$}&  & & \mbox{\large $-I$} &  && \mbox{\large $I$} &\\
\hline
&&&&&& 0& -b & 0\\
& \mbox{\large $-I$}&  & & \mbox{\large $(b+1)I$} &  & 0 & 0 & -b\\
&&&&&& -b & 0 & 0\\
\hline
&&& 0& 0& -b &  b+2c+1 & -c-1 & -c-1\\
& \mbox{\large $I$}&  &  -b&  0 & 0 & -c-1 & b+2c+1 & -c-1\\
&&& 0 & -b & 0 & -c-1 & -c-1 & b+2c+1\\
\end{array}
\right),
\]
where $I$ denotes the $3\times 3$ identity matrix.
Then a direct calculation shows $\det G_1=(3bc+6b+6c+5)^2$.
Since this value is positive, we have $\det \mathcal{L}=\det G_1$.

%%%%%%%%%%%%%%%%%%%%%%%%%%%%%%%%%%%%%%%%%%%%%%%%%%%%%%%%%%%%%%%%%%%%%%%%%%%%
\section{The case where $a=1$}\label{sec:a1}

The purpose of this section is to show that
the link $\mathcal{L}$ is quasi-alternating when $a=1$.
The link diagram $D$ is illustrated in Figure \ref{fig:white}.
For $i\in \{1,2,3\}$, let $c_i$ be the upper crossing of the white region $\alpha_i$.
Let $\varepsilon_i\in \{*,\infty,0\}$.
We use the notation $L(\varepsilon_1,\varepsilon_2,\varepsilon_3)$
to express the link obtained from the link diagram $D$
by performing a resolution of type $\varepsilon_i$ at the crossing $c_i$.
Here, if $\varepsilon_i=*$, then the crossing $c_i$ is not changed.
If $\varepsilon_i=\infty$ or $0$, then $c_i$ is split vertically or horizontally, respectively,
as in Figure \ref{fig:resolution}.

\begin{lemma}\label{lem:a1qa}
\begin{enumerate}
\item $L(0,0,*)=L(0,\infty,0)=P(b+c+1,b+c+1,b+c+1)$.  Hence these are alternating.
\item $L(0,\infty,\infty)=L(\infty,0,\infty)=L(\infty,\infty,0)$, and these are alternating.
\item $L(\infty,\infty,\infty)$ is quasi-alternating.
\end{enumerate}
\end{lemma}

\begin{proof}
(1) This is obvious from their diagrams.

(2) The equivalence of three links follows from the symmetry.
An alternating diagram of $L(0,\infty,\infty)$ is illustrated in Figure \ref{fig:a1-0ii}.

\begin{figure}[tb]
\includegraphics*[scale=0.4]{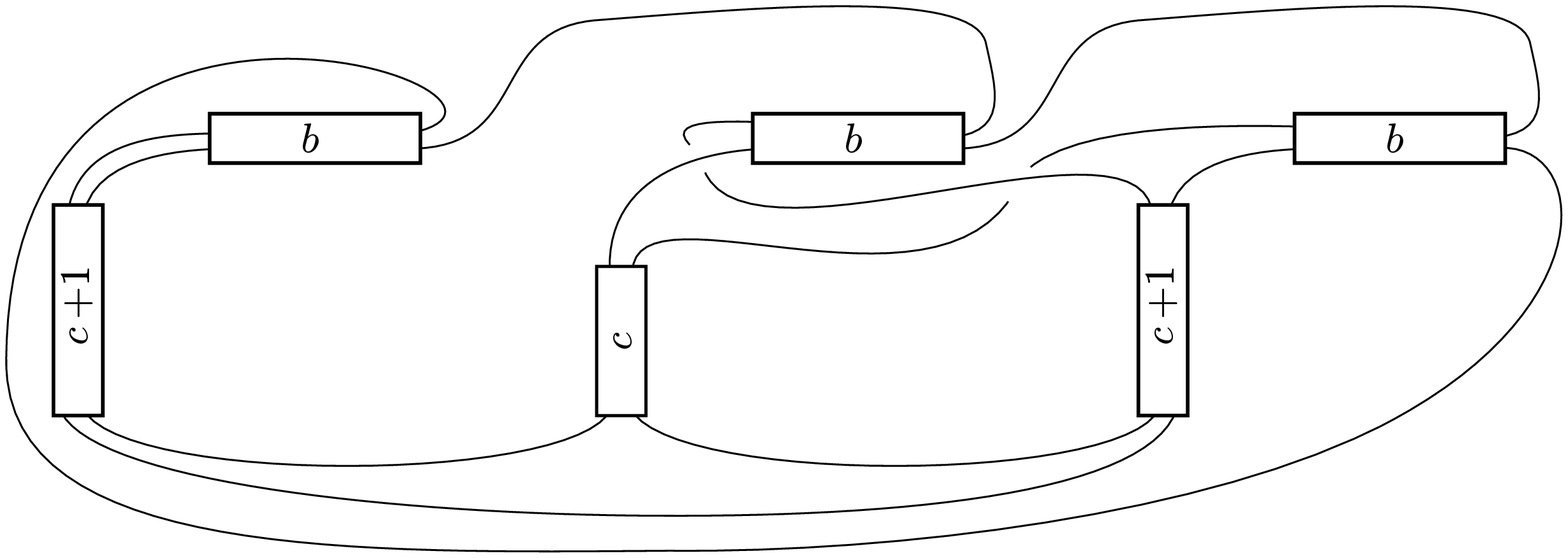}
\caption{$L(0,\infty,\infty)$ is alternating}\label{fig:a1-0ii}
\end{figure}

(3) The link $L(\infty,\infty,\infty)$ is equivalent to one as in Figure \ref{fig:a1-iii}.
This link is shown to be quasi-alternating by Peters \cite{P}.
\end{proof}

\begin{figure}[tb]
\includegraphics*[scale=0.4]{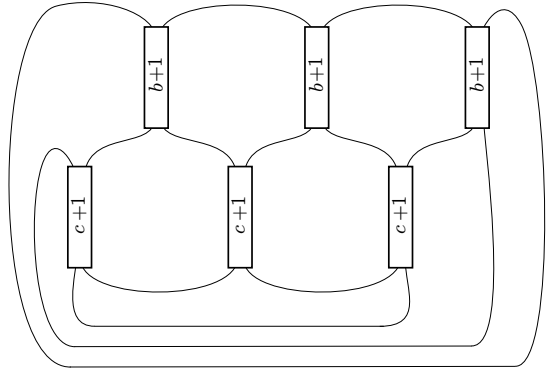}
\caption{$L(\infty,\infty,\infty)$}\label{fig:a1-iii}
\end{figure}

To conclude that $\mathcal{L}$ is quasi-alternating,
we need the values of determinants of some of links $L(\varepsilon_1,\varepsilon_2,\varepsilon_3)$.

Recall that the diagram  $D$ (Figure \ref{fig:white}) of $\mathcal{L}$ yields the Goeritz matrix $G_1$ described in Section \ref{sec:det}.
For $L(0,*,*)$ (resp. $L(\infty,*,*)$),
its diagram is obtained from $D$ by splitting the crossing $c_1$ horizontally (resp. vertically).
Then, the corresponding Goeritz matrix is 
obtained from $G_1$ by replacing the $(1,1)$-entry with $0$, or
deleting the first row and column, respectively. 
In this way, calculating determinants of the matrices gives Table \ref{table:det}.

\begin{table}[htb]
\begin{tabular}{c|c}
\rm{Link} & \rm{Determinant} \\
\hline \hline
$L(0,*,*)$ & $2(b+c+1)(3bc+6b+6c+5)$ \\
$L(\infty,*,*)$ & $(3bc+4b+4c+3)(3bc+6b+6c+5)$ \\
\hline
$L(0,0,*)$  &  $3(b+c+1)^2$  \\
$L(0,\infty,*)$ &  $(b+c+1)(6bc+9b+9c+7)$ \\
\hline
$L(\infty,0,*)$ &  $(b+c+1)(6bc+9b+9c+7)$ \\
$L(\infty,\infty,*)$  & $(3bc+3b+3c+2)(3bc+5b+5c+4)$ \\
\hline
$L(0,\infty,0)$  &  $3(b+c+1)^2$ \\
$L(0,\infty,\infty)$  & $2(b+c+1)(3bc+3b+3c+2)$ \\
\hline
$L(\infty,\infty,\infty)$ &  $(3bc+3b+3c+2)^2$ \\
\end{tabular}
\vspace{3mm}
\caption{Determinants of links}\label{table:det}
\end{table}

\begin{theorem}\label{thm:qa-a1}
Assume $a=1$.   Then the link $\mathcal{L}$ is quasi-alternating.
Furthermore, $L(0,*,*)$ and $L(\infty,*,*)$ are quasi-alternating.
\end{theorem}

\begin{proof}
By Lemma \ref{lem:a1qa}, 
$L(0,\infty,0)$ and $L(0,\infty,\infty)$ are alternating.
As shown in Table \ref{table:det},
we have $\det L(0,\infty,*) = \det L(0,\infty,0) + \det L(0,\infty,\infty)$.
Hence $L(0,\infty,*)$ is quasi-alternating.
Similarly, because $L(0,0,*)$ is alternating
and 
$\det L(0,*,*) = \det L(0,0,*) + \det L(0,\infty,*)$,
$L(0,*,*)$ is quasi-alternating.
Also, we can verify that $L(\infty,*,*)$ is quasi-alternating by the same argument.
Finally, 
the equation $\det \mathcal{L}= \det L(0,*,*) + \det L(\infty,*,*)$
implies the conclusion that $\mathcal{L}$ is quasi-alternating.
\end{proof}

%%%%%%%%%%%%%%%%%%%%%%%%%%%%%%%%%%%%%%%%%%%%%%%%%%%%%%%%%%%%%%%%%%%%%%%%%%%%%%%%%%%
\section{Induction}\label{sec:induction}

As in Section \ref{sec:a1},
we use the notation $L(a\colon \varepsilon_1,\varepsilon_2,\varepsilon_3)$ with
$\varepsilon\in \{*,\infty,0\}$
to denote the link obtained from $\mathcal{L}$ by performing the resolution
of type $\varepsilon_i$ at the crossing $c_i$.
Here, $c_i$ is located at the top of the white region $\alpha_i$.
See Figure \ref{fig:block}.
Because we will use an inductive argument, the parameter $a$ is added.
In particular, $\mathcal{L}=L(a\colon *,*,*)$.

\begin{lemma}\label{lem:diagram}
Suppose $a>1$.
\begin{enumerate}
\item $L(a\colon 0,0,*)=L(a\colon 0,\infty,0)=L(a\colon \infty,0,0)=P(b+c+1,b+c+1,b+c+1)$, and these are alternating.
\item $L(a\colon 0,\infty,\infty)=L(a\colon \infty,0,\infty)=L(a\colon \infty,\infty,0)=L(a-1\colon 0,*,*)$.
\item $L(a\colon \infty,\infty,\infty)=L(a-1\colon *,*,*)$.
\end{enumerate}
\end{lemma}

\begin{proof}
These immediately follow from the diagrams.
\end{proof}

\begin{lemma}\label{lem:cal}
For $\mathcal{L}$, $L(a\colon 0,*,*)$ and $L(a\colon \infty,*,*)$,
\begin{eqnarray*}
\det \mathcal{L} &=& (3ab+3bc+3ca+3a+3b+3c+2)^2,\\
\det L(a\colon  0,*,*) &= & 2(b+c+1)(3ab+3bc+3ca+3a+3b+3c+2),\\
\det L(a\colon  \infty,*,*) & = & (3ab+3bc+3ca+3a+b+c)(3ab+3bc+3ca+3a+3b+3c+2).
\end{eqnarray*}
Hence $\det \mathcal{L}=\det L(a\colon 0,*,*)+\det L(a\colon \infty,*,*)$.
\end{lemma}

\begin{proof}
Let $G$ be the $(3a+6)\times (3a+6)$ Goeritz matrix obtained from the link diagram $D$ of $\mathcal{L}$.
As in Section \ref{sec:det}, 

\[
G=\left (
\begin{array}{cccc|ccc}
 -2I & I &&&&& \\
I &  \ddots &&&&  & \\
&& -2I & I &&& \\
&& I   & -2I & I & O& O\\
\hline
&&     & I   & &  & \\
&&     &  O   &  &   \mbox{\Huge $G_1$} & \\
&&     &  O   &  &                     & \\
\end{array}
\right ),
\]
where $I$ is the $3\times 3$ identity matrix,
$O$ is the $3\times 3$ zero matrix, and $G_1$ is exactly the $9\times 9$ matrix given in Section \ref{sec:det}.
To calculate its determinant,
add the $i$-th column multiplied by $1/2$ to the $(i+3)$-th column for $i=1,2,3$.
Then reduce the matrix to a $(3a+3)\times (3a+3)$ matrix.
By repeating this process, we have $\det G=(-1)^{a-1}a^3 \det G_1'$, where
$G_1'$ is obtained from $G_1$ by replacing the upper left $3\times 3$ block $-I$
with $-\frac{1}{a}I$.
Thus $\det G=(-1)^{a-1}(3ab+3bc+3ca+3a+3b+3c+2)^2$, and so
$\det \mathcal{L}=(3ab+3bc+3ca+3a+3b+3c+2)^2$.

%%%

Consider the diagram of $L(a\colon  0,*,*)$ obtained from the diagram $D$ (Figure \ref{fig:block})
by splitting the crossing $c_1$ horizontally.
The corresponding Goeritz matrix $G_0$ is the above $G$ with replacing the $(1,1)$-entry with
$-1$.
Add the first column to the $4$-th column, and the $i$-th column multiplied by $1/2$ to
the $(i+3)$-th column for $i=2,3$.
Then reduce the size as before.
Repeating this gives $\det G_0=(-1)^{a-1}a^2 \det G_1''$, where
$G_1''$ is obtained from $G_1$ by replacing the $(i,i)$-entry with
$0$, $-1/a$, $-1/a$, respectively, for $i=1,2,3$.
Then we have $\det L(a\colon  0,*,*)=2(b+c+1)(3ab+3bc+3ca+3a+3b+3c+2)$.

%%%

Finally, the diagram of $L(a\colon  \infty,*,*)$ is obtained from $D$
by splitting the crossing $c_1$ vertically.
The corresponding Goeritz matrix $G_\infty$ is $G$ with deleting
the first column and row.
Add the $i$-th column multiplied by $1/2$ to the $(i+3)$-th column for $i=1,2$.
Then reduce the size of matrix.
Repeating this yields $\det G_\infty=(-1)^{a-3} (a-1)a^2 \det G_1'''$, where
$G_1'''$ is obtained from $G_1$ by replacing the $(i,i)$-entry with $-1/(a-1)$, $-1/a$, $-1/a$,
respectively, for $i=1,2,3$.
Then $\det L(a\colon  \infty,*,*) = (3ab+3bc+3ca+3a+b+c)(3ab+3bc+3ca+3a+3b+3c+2)$.
\end{proof}

By a similar process to the proof of Lemma \ref{lem:cal},
we can calculate of determinants of some other links, as in Table \ref{table:ind}.
We omit the details.

\begin{table}[htb]
\begin{tabular}{c|c}
\rm{Link} & \rm{Determinant} \\
\hline \hline
$L(a\colon 0,0,*)$ & $3(b+c+1)^2$ \\
$L(a\colon 0,\infty,*)$ & $(b+c+1)(6ab+6bc+6ca+6a+3b+3c+1)$ \\
$L(a\colon \infty,\infty,*)$ & $  (3ab+3bc+3ca+3a+2b+2c+1)(3ab+3bc+3ca+3a-1)$ \\
\end{tabular}
\vspace{3mm}
\caption{Determinants of links}\label{table:ind}
\end{table}

\begin{lemma}\label{lem:sum}
We have the following equations.
\begin{eqnarray*}
\det L(a\colon 0,*,*) &=& \det L(a\colon 0,0,*) +\det L(a\colon 0,\infty,*), \\
\det L(a\colon 0,\infty,*) &=& \det L(a\colon 0,\infty,0) +\det L(a\colon 0,\infty,\infty), \\
\det L(a\colon \infty,*,*) &=& \det L(a\colon \infty,0,*) +\det L(a\colon \infty,\infty,*), \\
\det L(a\colon \infty,0,*) &=& \det L(a\colon \infty,0,0) +\det L(a\colon \infty,0,\infty), \\
\det L(a\colon \infty,\infty,*) &=& \det L(a\colon \infty,\infty,0) +\det L(a\colon \infty,\infty,\infty).
\end{eqnarray*}
\end{lemma}

\begin{proof}
These immediately follow from Lemmas \ref{lem:diagram}, \ref{lem:cal} and Table \ref{table:ind}.
\end{proof}

%%%%%%%%%%%%%%%%%%%%%%%%%%%%%%%%%%%%%%%%%%%%%%%%%%%%%%%%%%%%%

\begin{proof}[Proof of Theorem \ref{thm:qa}]
We prove a stronger claim that not only $\mathcal{L}=L(a\colon *,*,*)$ but
$L(a\colon 0,*,*)$ is quasi-alternating.
The proof is done by induction on $a$.
By Theorem \ref{thm:qa-a1}, the claim is true when $a=1$.
Suppose $a>1$ and that the claim holds for $a-1$.

By Lemma \ref{lem:cal},
if both $L(a\colon 0,*,*)$ and $L(a\colon \infty,*,*)$ are quasi-alternating,
then $\mathcal{L}$ is quasi-alternating.

First, consider $L(a\colon 0,*,*)$.
By the resolution at the crossing $c_2$,
we obtain $L(a\colon 0,0,*)$ and $L(a\colon 0,\infty,*)$.
For the latter, perform the resolution at the crossing $c_3$ to yield
$L(a\colon 0,\infty,0)$ and $L(a\colon 0,\infty,\infty)$.
Then the claim that $L(a\colon 0,*,*)$ is quasi-alternating
follows from the facts
that $L(a\colon 0,0,*)$ and $L(a\colon 0,\infty,0)$ are alternating (Lemma \ref{lem:diagram})
and $L(a\colon 0,\infty,\infty)\ (=L(a-1\colon 0,*,*))$ is quasi-alternating by our inductive assumption,
coupled with the equations among determinants (Lemma \ref{lem:sum}).
Similarly, we can show that $L(a\colon \infty,*,*)$ is quasi-alternating.
\end{proof}

%%%%%%%%%%%%%%%%%%%%%%%%%%%%%%%%%%%%%%%%%%%%%%%%%%%%%%%%
\section{Remarks}\label{sec:rem}

(1) Boyer, Gordon and Watson \cite{BGW} propose a conjecture that
an irreducible rational homology $3$-sphere is an $L$-space if and only
if its fundamental group is not left-orderable.
For $K=P(2a+1,2b+1,2c+1)$,
$\pi_1 \Sigma_2(K)$ is not left-orderable, since $\Sigma_2(K)$ is a Seifert-fibered $L$-space
(\cite{BGW}).
By Theorem \ref{thm:main}, $\Sigma_3(K)$ is also an $L$-space.
Hence it is an interesting task to show that $\pi_1 \Sigma_3(K)$ is not left-orderable.

(2)
Among genus one pretzel knots, for example,
$P(-3,5,5)$ is non-alternating.
It is known that its double branched cover is not an $L$-space (\cite{CK,Gr}).
Thus we may not expect that the $3$-fold cyclic branched cover is an $L$-space.

(3) Let $K$ be a pretzel knot $P(3,3,-n)$ with $n\ge 3$, odd.
If $n>3$, then $K$ is quasi-alternating, but $P(3,3,-3)$, which is $9_{46}$ in the knot table,
is not quasi-alternating (see \cite{CK,Gr}).
Nevertheless, $\Sigma_2(K)$ is always an $L$-space.
By a similar argument, we can show that $\Sigma_3(K)$ is an $L$-space, but
the details will be treated elsewhere.

(4) For an alternating pretzel knot $P(2a+1,2b+1,2c+1)$,
we may expect that the $d$-fold cyclic branched cover is an $L$-space for 
at least small $d\ge 4$.

%%%%%%%%%%%%%%%%%%%%%%%%%%%%%%%%%%%%%%%%%%%%%%%%%%%%%%%%%%%%%%%%%%%%%%
\bibliographystyle{amsplain}

\begin{thebibliography}{BGW}

\bibitem{B}
J. Baldwin,
\textit{Heegaard Floer homology and genus one, one-boundary component open books},
J. Topol. \textbf{1} (2008), 963--992.

%\bibitem{BP}
%M. Boileau and J. Porti,
%\textit{Geometrization of $3$-orbifolds of cyclic type},
%Asterisque No. 272 (2001).

\bibitem{BGW}
S. Boyer, C. McA. Gordon and L. Watson,
\textit{On $L$-spaces and left-orderable fundamental groups},
Math. Ann. \textbf{356} (2013), 1213--1245.

\bibitem{BZ}
S. Boyer and X. Zhang,
\textit{Cyclic surgery and boundary slopes},
Geometric topology (Athens, GA, 1993), 62--79, AMS/IP Stud. Adv. Math., 2.1, Amer. Math. Soc., Providence, RI, 1997. 


%\bibitem{BRW}
%S. Boyer, D. Rolfsen and B. Wiest,
%\textit{Orderable $3$-manifold groups},
%Ann. Inst. Fourier (Grenoble) \textbf{55} (2005), 243--288.

%\bibitem{BW}
%M. Brittenham and  Y. Q. Wu,
%\textit{The classification of exceptional Dehn surgeries on 2-bridge knots},
%Comm. Anal. Geom. \textbf{9} (2001), 97--113.

\bibitem{BuZ}
G. Burde and H. Zieschang,
\textit{Knots}, de Gruyter Studies in Mathematics, 5. Walter de Gruyter \& Co., Berlin, 2003.

\bibitem{CK}
A. Champanerkar and I. Kofman,
\textit{Twisting quasi-alternating links},
Proc. Amer. Math. Soc. \textbf{137} (2009), 2451--2458.


%\bibitem{CLW}
%A. Clay, T. Lidman and L. Watson,
%\textit{Graph manifolds, left-orderability and amalgamation},
%preprint, \texttt{arXiv:1106.0486}.

%\bibitem{CT}
%A. Clay and M. Teragaito,
%\textit{Left-orderability and exceptional Dehn surgery on two-bridge knots},
%to appear in the Proceedings of Geometry and Topology Down Under,
%Contemporary Mathematics Series.

%\bibitem{DPT}
%M. D\k{a}bkowski, J. Przytycki and A. Togha,
%\textit{Non-left-orderable $3$-manifold groups},
%Canad. Math. Bull. \textbf{48} (2005), 32--40.

\bibitem{Ga}
D. Gabai,
\textit{Genera of the arborescent links},
Mem. Amer. Math. Soc. \textbf{59} (1986), no. 339, 1--98. 

%\bibitem{G}
%J. Greene,
%\textit{A spanning tree model for the Heegaard Floer homology
%of a branched double-cover}, 
%J. Topology \textbf{6} (2013), 525--567.

\bibitem{Gr}
J. Greene, 
\textit{Homologically thin, non-quasi-alternating links},
Math. Res. Lett. \textbf{17} (2010), 39--49. 


%\bibitem{HT}
%A. Hatcher and W. Thurston,
%\textit{Incompressible surfaces in 2-bridge knot complements},
%Invent. Math. \textbf{79} (1985), 225--246.

\bibitem{H}
M. Hori, a private communication.

\bibitem{Hu}
Y. Hu,
\textit{The left-orderability and the cyclic branched coverings}, preprint. 

%\bibitem{MO}
%C. Manolescu and P. Ozsv\'{a}th,
%\textit{On the Khovanov and knot Floer homologies of quasi-alternating links},
%Proceedings of Gokova Geometry-Topology Conference 2007, 60--81, Gokova Geometry/Topology %Conference (GGT), Gokova, 2008.

%\bibitem{Me}
%W. Menasco,
%\textit{Closed incompressible surfaces in alternating knot and link complements},
%Topology \textbf{23} (1984), 37--44. 

%\bibitem{Mi}
%J. Milnor, 
%\textit{On the $3$-dimensional Brieskorn manifolds $M(p,q,r)$},
%Knots, groups, and $3$-manifolds (Papers dedicated to the memory of R. H. Fox), 175--225.
%Ann. of Math. Studies, No. 84, Princeton Univ. Press, Princeton, N. J., 1975. 


%\bibitem{M}
%J. Montesinos,
%\textit{Surgery on links and double branched covers of $S^3$},
%Knots, groups, and $3$-manifolds (Papers dedicated to the memory of R. H. Fox), pp. 227--259.
%Ann. of Math. Studies, No. 84, Princeton Univ. Press, Princeton, N.J., 1975.



\bibitem{OS1}
P. Ozsv\'{a}th and Z. Szab\'{o},
\textit{On knot Floer homology and lens space surgeries},
Topology \textbf{44} (2005), 1281--1300. 

\bibitem{OS2}
P. Ozsv\'{a}th and Z. Szab\'{o},
\textit{On the Heegaard Floer homology of branched double-covers},
Adv. Math. \textbf{194} (2005), 1--33.

\bibitem{Pa}
R. Patton,
\textit{Incompressible punctured tori in the complements of alternating knots}, 
Math. Ann. \textbf{301} (1995), 1--22. 

\bibitem{P}
T. Peters,
\textit{On $L$-spaces and non left-orderable $3$-manifold groups}, preprint.

%\bibitem{R}
%D. Rolfsen,
%\textit{Knots and links},
%Mathematics Lecture Series, No. 7. Publish or Perish, Inc., Berkeley, Calif., 1976.

\bibitem{T}
M. Teragaito,
\textit{Four-fold cyclic branched covers of genus one two-bridge knots are $L$-spaces},
to appear in Bol. Soc. Mat. Mexicana.

\bibitem{Tr}
A. T. Tran,
\textit{On left-orderability and cyclic branched coverings},
preprint.

\end{thebibliography}

\end{document}